\documentclass[final,3p]{elsarticle}
\usepackage[latin1]{inputenc}
 \usepackage{graphics}
 \usepackage{graphicx}
 \usepackage{epsfig}
\usepackage{amssymb}
 \usepackage{amsthm}
 \usepackage{lineno}
 \usepackage{amsmath}
\usepackage{color}
   \numberwithin{equation}{section}
\usepackage{mathrsfs}

\journal{ } 
\date{} 

\makeatletter \def\ps@pprintTitle{ 
\let\@oddhead\@empty 
\let\@evenhead\@empty 
\let\@oddfoot\@empty 
\let\@evenfoot\@oddfoot } 
\makeatother 

\newtheorem{thm}{Theorem}[section]

\newtheorem{lem}[thm]{Lemma}
\newtheorem{prop}[thm]{Proposition}
\newtheorem{defn}[thm]{Definition}

\newtheorem{prope}[thm]{Property}

\biboptions{sort&compress,square}
\allowdisplaybreaks
\begin{document}
\begin{frontmatter}
\author{Hongfeng Li$^{1}$}
\ead{lihf728@nenu.edu.cn}
\author{Kefeng Liu$^{2}$}
\ead{kefeng@cqut.edu.cn}
\author{Yong Wang$^{1,*}$}
\ead{wangy581@nenu.edu.cn}
\cortext[cor]{Corresponding author.}

\address{$^1$School of Mathematics and Statistics, Northeast Normal University, Changchun, 130024, China}
\address{$^2$Mathematical Science Research Center, Chongqing University of Technology, Chongqing, 400054, China}

\title{The noncommutative residue and sub-Riemannian limits\\ for the twisted BCV spaces}
\begin{abstract}
In this paper, we derive the sub-Riemannian version of the Kastler-Kalau-Walze type theorem and the Dabrowski-Sitarz-Zalecki type theorem for the twisted BCV spaces. We also compute the Connes conformal invariants for the twisted product, as well as the sub-Riemannian limits of the Connes conformal invariants for the twisted BCV spaces.
\end{abstract}
\end{frontmatter}
\section{ Introduction}
\indent The (singly) warped product $B\times_bF$  of two pseudo-Riemannian manifolds $(B,g_B)$ and
    $(F,g_F)$ with a smooth function $b:B\rightarrow (0,\infty)$ is the product
    manifold $B\times F$ with the metric tensor $$g=g_B\oplus
    b^2g_F.$$ Here, $(B,g_B)$ is called the base manifold and
    $(F,g_F)$ is called as the fiber manifold and $b$ is called as
    the warping function. Generalized Robertson-Walker space-times
    and standard static space-times are two well-known warped
    product spaces. The concept of warped products was first introduced by
    Bishop and O'Neil (see \cite{BO}) to construct examples of Riemannian
    manifolds with negative curvature. In Riemannian geometry,
    warped product manifolds and their generic forms have been used
    to construct new examples with interesting curvature properties
    since then. In \cite{DD}, Dobarro and Dozo had studied from the viewpoint of partial differential equations and variational methods,
    the problem of showing when a Riemannian metric of constant scalar curvature can be produced on a product manifolds by a warped product construction.
    In \cite{EJK}, Ehrlich et al. got explicit solutions to warping function to have a constant scalar curvature for generalized Robertson-Walker space-times.
    In \cite{ARS}, explicit solutions were also obtained for the warping
    function to make the space-time as Einstein when the fiber is
    also Einstein.\\
    \indent Singly warped products have two natural generalizations. A doubly warped product
      $(M,g)$ is a
      product manifold of form
      $M=_fB\times_{b}F$, with smooth functions $$b:\,B\rightarrow (0,\infty), \ \ f:\,F\rightarrow (0,\infty)$$
      and the metric tensor $g=f^2g_B\oplus b^2g_F.$ In \cite{U2}, \"{U}nal studied geodesic completeness of Riemannian
      doubly warped products and Lorentzian doubly warped products. A twisted product
      $(M,g)$ is a
      product manifold of form
      $M=B\times_{b}F$, with a smooth function $b:B\times F\rightarrow (0,\infty)$,
      and the metric tensor $g=g_B\oplus b^2g_F.$ In \cite{FGKU}, they
      showed that mixed Ricci-flat twisted products could be expressed as warped
      products. As a consequence, any Einstein twisted products are
      warped products.\\
\indent  In \cite{Co1}, Connes used the noncommutative residue to derive a conformal 4-dimensional Polyakov action analogy. Connes showed us that the noncommutative residue on a compact manifold $M$ coincided with the Dixmier's trace on pseudo-differential operators of order $-{\rm {dim}}M$ \cite{Co2}.
And Connes claimed that the noncommutative residue of the square of the inverse of the Dirac operator was proportioned to the Einstein-Hilbert action.  Kastler \cite{Ka} gave a
brute-force proof of this theorem. Kalau and Walze proved this theorem in the normal coordinates system simultaneously \cite{KW} .
Ackermann proved that
the Wodzicki residue  of the square of the inverse of the Dirac operator ${\rm  Wres}(D^{-2})$ in turn is essentially the second coefficient
of the heat kernel expansion of $D^{2}$ \cite{Ac}. On the other hand, Wang generalized the Connes' results to the case of manifolds with boundary \cite{Wa1,Wa2},
and proved the Kastler-Kalau-Walze type theorem for the Dirac operator and the signature operator on lower-dimensional manifolds
with boundary \cite{Wa3}. In \cite{Wa3,Wa4}, Wang computed $$\widetilde{{\rm Wres}}[\pi^+D^{-1}\circ\pi^+D^{-1}]$$ and $$\widetilde{{\rm Wres}}[\pi^+D^{-2}\circ\pi^+D^{-2}],$$ where the two operators are symmetric, in these cases the boundary term vanished. But for $$\widetilde{{\rm Wres}}[\pi^+D^{-1}\circ\pi^+D^{-3}],$$ Wang and Wang got a nonvanishing boundary term \cite{Wa5}, and gave a theoretical explanation for gravitational action on boundary. In others words, Wang provided a kind of method to study the Kastler-Kalau-Walze type theorem for manifolds with boundary. In \cite{DL}, Dabrowski et al. defined bilinear functionals of vector fields and differential forms, the densities of which yielded the  metric and Einstein tensors on even dimensional Riemannian manifolds. In \cite{WW2}, the authors got the spectral Einstein functional
 associated with Dirac operators with torsion on compact manifolds with  boundary. For lower dimensional compact Riemannian manifolds
  with  boundary, they computed the lower dimensional  residue of $$\widetilde{\nabla}_{X}\widetilde{\nabla}_{Y}D_{T}^{-4}$$ and
   got the Dabrowski-Sitarz-Zalecki type theorem. In \cite{Wu1}, Wu and Wang defined the spectral Einstein functional associated with the Dirac operator for manifolds with boundary, and computed the noncommutative residue $$\widetilde{{\rm Wres}}[\pi^+(\nabla_X^{S(TM)}\nabla_Y^{S(TM)}D^{-1})\circ\pi^+(D^{-(n-1)})]$$ and $$\widetilde{{\rm Wres}}[\pi^+(\nabla_X^{S(TM)}\nabla_Y^{S(TM)}D^{-2})\circ\pi^+(D^{-(n-2)})]$$ on $n$-dimensional compact manifolds, $n$ is even. They also computed $$\widetilde{{\rm Wres}}[\pi^+(\nabla_X^{S(TM)}\nabla_Y^{S(TM)}D^{-2})\circ\pi^+(D^{-(n-1)})]$$ for $n$-dimensional manifolds with boundary, $n$ is odd. Hence, we get a method to study the Dabrowski-Sitarz-Zalecki type theorem for manifolds with boundary. For an even dimensional, compact, conformal manifold without boundary, Ugalde constructed a conformally invariant differential operator of order the dimension of the manifold \cite{UM}. In the conformally flat case, this operator coincides with the critical GJMS
operator of Graham-Jenne-Mason-Sparling. The author used the Wodzicki residue of a pseudo-differential operator of order $-2$, originally defined by Connes, acting on middle
dimensional forms.\\
\indent In \cite{LiuW}, Liu and Wang computed the adiabatic limit of the scalar curvature explicitly by some formulas \cite{LiuZ}, they used this result to give a new vanishing theorem
for general foliations with spin leave. For a manifold $M$ with splitting tangent bundle $TM=F\oplus F^\perp$ and $F$ being not integrable, they also computed the adiabatic limit of the scalar
curvature similarly, then some extra singular terms $O(\frac{1}{\varepsilon})$ would appear. They considerded the adiabatic limit of $\varepsilon D_F^2$ and got a vanishing theorem. By the Bismut-Kodiara-Nakano formula and taking adiabatic limits, a vanishing theorem can be got for complex foliations. Meanwhile, a Kastler-Kalau-Walze theorem analogy for foliations with spin leave was given. In \cite{WW1},  Wang and Wei computed sub-Riemannian limits of Gaussian curvature for a Euclidean $C^2$-smooth surface in the BCV spaces and the twisted Heisenberg group away from characteristic points and signed geodesic curvature for Euclidean $C^2$-smooth curves on surfaces. They got the sub-Riemannian Gauss-Bonnet theorem for BCV spaces.

 The motivation of this paper is to get the sub-Riemannian version of the Kastler-Kalau-Walze type theorem and the Dabrowski-Sitarz-Zalecki type theorem for the twisted BCV spaces. We compute the Connes conformal invariants for the twisted product. We also compute the sub-Riemannian limit of the Connes conformal invariants for the twisted BCV spaces. In \cite{WW1} and related some works on sub-Riemannian Gauss-Bonnet theorem, the sub-Riemannian limit on only two-dimensional surfaces was computed. In order to get the sub-Riemannian version of the Kastler-Kalau-Walze type theorem and the Dabrowski-Sitarz-Zalecki type theorem, we compute the sub-Riemannian limit for the higher dimensional manifolds i.e. twisted BCV spaces for the first time.

\indent The paper is organized in the following way. In Section 2, we compute the sub-Riemannian limit of scalar curvature of twisted BCV spaces and get the sub-Riemannian version of the Kastler-Kalau-Walze type theorem for twisted BCV spaces. In Section 3, we compute sub-Riemannian limits of the Einstein functional on the twisted BCV spaces and get the sub-Riemannian version of the Dabrowski-Sitarz-Zalecki type theorem for twisted BCV spaces. In Section 4, we compute the Connes conformal invariants for the twisted product and also compute the sub-Riemannian limit of the Connes conformal invariants for the twisted BCV spaces.


\section{ The sub-Riemannian Kastler-Kalau-Walze type theorem for the twisted BCV spaces }
We first introduce some notations on the BCV spaces. Let $\lambda$ and $\tau$ be two real numbers with $\tau>0$. The BCV spaces $N$ is the set
$$\{(x_1,x_2,x_3)\in \mathbb{R}^3|1+\frac{\lambda}{4}(x_1^2+x_2^2)>0\}.$$
 Let
\begin{equation}
X_1=[1+\frac{\lambda}{4}(x_1^2+x_2^2)]\partial_{x_1}-\tau x_2 \partial_{x_3}, ~~X_2=[1+\frac{\lambda}{4}(x_1^2+x_2^2)]\partial_{x_2}+\tau x_1 \partial_{x_3},~~X_3=\partial_{x_3}.
\end{equation}
Then
\begin{equation}
\partial_{x_1}=\frac{1}{1+\frac{\lambda}{4}(x_1^2+x_2^2)}(X_1+\tau x_2X_3),~~\partial_{x_2}=\frac{1}{1+\frac{\lambda}{4}(x_1^2+x_2^2)}(X_2-\tau x_1X_3),~~\partial_{x_3}=X_3,
\end{equation}
and ${\rm span}\{X_1,X_2,X_3\}=TN.$ 

Let $H={\rm span}\{X_1,X_2\}$ be the horizontal distribution on $N$.
Let
\begin{equation}
\omega_1=\frac{dx_1}{1+\frac{\lambda}{4}(x_1^2+x_2^2)},~~\omega_2=\frac{dx_2}{1+\frac{\lambda}{4}(x_1^2+x_2^2)},~~\omega:
=\omega_3=dx_3+\tau\frac{x_2dx_1-x_1dx_2}{1+\frac{\lambda}{4}(x_1^2+x_2^2)}.
\end{equation}
Then $H={\rm Ker}\omega$. For the constant $L>0$, let
$$g_L=\omega_1\otimes \omega_1+\omega_2\otimes \omega_2+L\omega\otimes \omega,~~g=g_1$$ be the Riemannian metric on $N$. Then 
$$X_1,X_2,\widetilde{X_3}:=L^{-\frac{1}{2}}X_3$$ are orthonormal basis on $N$ with respect to $g_L$. We have
\begin{equation}\label{a20}
[X_1,X_2]=-\frac{\lambda}{2}x_2X_1+\frac{\lambda}{2}x_1X_2+2\tau X_3
,~~[X_2,X_3]=0,~~[X_1,X_3]=0.
\end{equation}
For $\tau>0$, $H$ is not integrable. 

The complete classification of BCV spaces for $\tau>0$ is as follow (see p.753 in \cite{FMV}): if $\lambda>0$ and $\tau>0$, then $$(N,g)\cong SU(2)\backslash \{\infty\}; $$ if $\lambda<0$ and $\tau>0$, then $$(N,g)\cong \widetilde{SL}(2,\mathbb{R});$$
if $\lambda=0$ and $\tau>0$, then $$(N,g)\cong \mathbb{H}^1.$$
Let $\nabla^L$ be the Levi-Civita connection on $N$ with respect to $g_L$. By the Koszul formula, we have
\begin{equation}\label{a21}
2\langle \nabla^L_{X_i}X_j,X_k\rangle_L=\langle[X_i,X_j],X_k\rangle_L-\langle[X_j,X_k],X_i\rangle_L+\langle[X_k,X_i],X_j\rangle_L,
\end{equation}
where $i,j,k=1,2,3$. By (\ref{a20}) and (\ref{a21}), we have
\begin{lem}\label{lemma1}
Let $N$ be the BCV spaces, then
\begin{align}
&\nabla^L_{X_1}X_1=\frac{\lambda}{2}x_2X_2,~~~\nabla^L_{X_1}X_2=-\frac{\lambda}{2}x_2X_1+\tau X_3,~~~\nabla^L_{X_2}X_1=-\frac{\lambda}{2}x_1X_2-\tau X_3,\nonumber\\
&\nabla^L_{X_1}X_3=-\tau LX_2,~~~\nabla^L_{X_3}X_1=-\tau LX_2,~~~\nabla^L_{X_2}X_2=\frac{\lambda}{2}x_1X_1,\nonumber\\
& \nabla^L_{X_2}X_3=\nabla^L_{X_3}X_2=\tau LX_1,~~~\nabla^L_{X_3}X_3=0.
\end{align}
\end{lem}
We recall the twisted product and its properties.
\begin{defn}
A twisted product $(M,g)$ is a product manifold of form $M=B\times_{f}F$ with the metric $g=g_B\oplus f^2g_F$, where $f:~B\times F\rightarrow (0,\infty)$ is smooth.\\
\indent Here, $(B,g_B)$ is called the base manifold and $(F,g_F)$ is called as the fiber manifold and $f$ is called as the twisted function.
\end{defn}
\begin{prop}\label{prop2}
Let $M=B\times_{f}F$ be a twisted product and let $X,Y\in \Gamma(TB)$ and $U,W\in \Gamma(TF)$. Then\\
      \noindent$(1)~~\nabla_XY=\nabla^B_XY;$\\
      \noindent$(2)~~\nabla_XU=\nabla_UX=\frac{X(f)}{f}U;$\\
      \noindent$(3)~~\nabla_UW=U({\rm ln}f)W+W({\rm ln}f)U-\frac{g_{F}(U,W)}{f}{\rm
      grad}_{F}f-fg_{F}(U,W){\rm
      grad}_{B}f+\nabla^{F}_UW.$
\end{prop}
\begin{prop}\label{prop3}
Let $M=B\times_{f}F$ be a twisted product and let $\omega_1\in \Gamma(T^*B)$, $\omega_2\in \Gamma(T^*F)$ and $X\in \Gamma(TB)$, $U\in \Gamma(TF)$. Let $\nabla^*$ be dual connection on $\Gamma(T^*M)$. Then\\
      \noindent$(1)~~\nabla_X^*\omega_1=\nabla^{B,*}_X\omega_1;$\\
      \noindent$(2)~~\nabla_X^*\omega_2=-\frac{X(f)}{f}\omega_2;$\\
      \noindent$(3)~~\nabla_U^*\omega_1=f(\omega_1,{\rm grad}_{B}f)U^*;$\\
      \noindent$(4)~~\nabla_U^*\omega_2=-U({\rm ln}f)\omega_2-(\omega_2,U)({\rm grad}_{F}({\rm ln}f))^*-\frac{(\omega_2,U)}{f}({\rm grad}_{B}f)^*+\frac{1}{f}(\omega_2,{\rm grad}_{F}f)U^*+\nabla^{F,*}_U\omega_2.$
\end{prop}
\begin{proof} Let $\omega_1\in \Gamma(T^*B)$, $\omega_2\in \Gamma(T^*F)$ and $X,Y\in \Gamma(TB)$, $U,W\in \Gamma(TF)$. Then by the definition of the dual connection, we have
\begin{align*}
(\nabla_X^*\omega_1,Y)+(\omega_1,\nabla_XY)=X(\omega_1,Y),
\end{align*}
and
\begin{align}\label{a32}
(\nabla_X^{B,*}\omega_1,Y)+(\omega_1,\nabla_X^BY)=X(\omega_1,Y).
\end{align}
Also, we have
\begin{equation*}
(\nabla_X^*\omega_1,U)+(\omega_1,\nabla_XU)=X(\omega_1,U).
\end{equation*}
By Proposition \ref{prop2}, we have
\begin{equation}\label{a33}
(\nabla_X^*\omega_1,U)+(\omega_1,\frac{X(f)}{f}U)=X(\omega_1,U).
\end{equation}
Then by (\ref{a32}) and (\ref{a33}), we get
\begin{equation}
\nabla_X^*\omega_1=\nabla^{B,*}_X\omega_1.
\end{equation}
Similarly, we can prove other equalities in Proposition \ref{prop3}.
\end{proof}
    Define the curvature, Ricci curvature and scalar curvature as
    follows:
    $$R(X,Y)Z=\nabla_X\nabla_Y-\nabla_Y\nabla_X-\nabla_{[X,Y]},$$
    $$ ~{\rm
    Ric}(X,Y)=\sum_k<R(X,E_k)Y,E_k>, ~~S=\sum_k{\rm
    Ric}(E_k,E_k),$$ where $E_k$ is a orthonormal base of $M$ with
    $<E_k,E_k>=1.$ The Hessian of $f$ is defined by
    $$H^f(X,Y)=XYf-(\nabla_XY)f.$$
\begin{prop}
Let $M=B\times_{f}F$ be a twisted product and let $X,Y,Z\in \Gamma(TB)$
      and $U,V,W\in \Gamma(TF)$. Then\\
\noindent $(1)R(X,Y)Z=R^B(X,Y)Z;$\\
\noindent $(2)R(V,X)Y=-\frac{H^{f}_B(X,Y)}{f}V;$\\
\noindent $(3)R(X,Y)V=0;$\\
\noindent $(4)R(V,W)X=VX({\rm ln}f)W-WX({\rm ln}f)V;$\\
\noindent $(5)R(X,V)W=-\frac{g(V,W)}{f}\nabla_X^B({\rm
grad}_Bf)+[WX({\rm ln}f)]V-g_{F}(W,V){\rm grad}_{F}(X({\rm
ln}f));$\\
\noindent $(6)R(V,W)U=g(V,U){\rm grad}_B(W({\rm ln}f))-g(W,U){\rm
grad}_B(V({\rm ln}f))+R^{F}(V,W)U
-\frac{|{\rm grad}_Bf|^2_B}{f^2}(g(W,U)V-g(V,U)W).$
\end{prop}
\begin{prop}\label{prop1}
Let $M=B\times_{f}F$ be a twisted product and let $X,Y\in \Gamma(TB)$
      and $V,W\in \Gamma(TF)$. Then\\
\noindent $(1) {\rm Ric} (X,Y)={\rm
Ric}^B(X,Y)+\frac{\widetilde{l}}{f}H_B^{f}(X,Y);$\\
\noindent $(2) {\rm Ric} (X,V)={\rm Ric}
(V,X)=(\widetilde{l}-1)[VX({\rm ln}f)];$\\
\noindent $(3) {\rm Ric} (V,W)={\rm Ric}^{F}
(V,W)+\left[\frac{\triangle_Bf}{f}+(\widetilde{l}-1)\frac{|{\rm
grad}_Bf|^2_B}{f^2}\right]g(V,W),$\\
\noindent where $\widetilde{l}$ is the sum of the dimensions of $B$ and $F$.
\end{prop}
\begin{prop}
Let $M=B\times_{f}F$ be a twisted product, then the scalar
    curvature $S$ has the following expression:\\
\begin{equation}\label{a1}
S=S^B+2\frac{\widetilde{l}}{f}\triangle_Bf+\frac{S^{F}}{f^2}+\widetilde{l}(\widetilde{l}-1)\frac{|{\rm grad}_Bf|^2_B}{f^2}.
\end{equation}
\end{prop}
\indent We will say that a surface $\Sigma\subset(N,g_L)$ is regular if $\Sigma$ is a Euclidean $C^2$-smooth compact and oriented surface. In particular we will assume that there exists
a Euclidean $C^2$-smooth function $u:N\rightarrow \mathbb{R}$ such that
$$\Sigma=\{(x_1,x_2,x_3)\in N:u(x_1,x_2,x_3)=0\}$$
and $$u_{x_1}\partial_{x_1}+u_{x_2}\partial_{x_2}+u_{x_3}\partial_{x_3}\neq 0.$$ Let $$\nabla_Hu=X_1(u)X_1+X_2(u)X_2.$$
 A point $x\in\Sigma$ is called {\it characteristic} if $\nabla_Hu(x)=0$.
 We define the characteristic set $$C(\Sigma):=\{x\in\Sigma|\nabla_Hu(x)=0\}.$$
  Our computations will
be local and away from characteristic points of $\Sigma$. Let us define first
$$p:=X_1u,~~q:=X_2u ,~~{\rm and}~~r:=\widetilde{X}_3u.$$
We then define
\begin{align}
&l:=\sqrt{p^2+q^2},~~~~l_L:=\sqrt{p^2+q^2+r^2},~~~~\overline{p}:=\frac{p}{l},\nonumber\\
&\overline{q}:=\frac{q}{l},~~~~
\overline{p_L}:=\frac{p}{l_L},~~~~\overline{q_L}:=\frac{q}{l_L},~~~~\overline{r_L}:=\frac{r}{l_L}.
\end{align}
In particular, we have $$\overline{p}^2+\overline{q}^2=1.$$ These functions are well defined at every non-characteristic point. Let
\begin{align}\label{a23}
v_L=\overline{p_L}X_1+\overline{q_L}X_2+\overline{r_L}\widetilde{X_3},~~~~\overline{e}_1=\overline{q}X_1-\overline{p}X_2,~~~~
\overline{e}_2=\overline{r_L}~~\overline{p}X_1+\overline{r_L}~~ \overline{q}X_2-\frac{l}{l_L}\widetilde{X_3},
\end{align}
then $v_L$ is the Riemannian unit normal vector to $\Sigma$ and $\overline{e}_1,\overline{e}_2$ are the orthonormal basis of $\Sigma$. On $T\Sigma$ we define a linear transformation $J_L:T\Sigma\rightarrow T\Sigma$ such that
\begin{equation}
J_L(\overline{e}_1):=\overline{e}_2;~~~~J_L(\overline{e}_2):=-\overline{e}_1.
\end{equation}
In the following, we compute the sub-Riemannian limit of the Riemannian Gaussian curvature of surfaces in the BCV spaces. We define the {\it second fundamental form} $II^L$ of the
embedding of $\Sigma$ into $(N,g_L)$:
\begin{equation}
II^L=\left(
  \begin{array}{cc}
   \langle \nabla^{L}_{\overline{e}_1}v_L,\overline{e}_1\rangle_{L},
    & \langle \nabla^{L}_{\overline{e}_1}v_L,\overline{e}_2\rangle_{L} \\
   \langle \nabla^{L}_{\overline{e}_2}v_L,\overline{e}_1\rangle_{L},
    & \langle \nabla^{L}_{\overline{e}_2}v_L,\overline{e}_2\rangle_{L} \\
  \end{array}
\right).
\end{equation}
Similarly to Theorem 4.3 in \cite{CDPT}, we have
\begin{thm} The second fundamental form $II^L$ of the
embedding of $\Sigma$ into $(N,g_L)$ is given by
\begin{equation}\label{a24}
II^L=\left(
  \begin{array}{cc}
   h_{11}, & h_{12}\\
   h_{21}, & h_{22}\\
  \end{array}
\right),
\end{equation}
where $$h_{11}= \frac{l}{l_L}[X_1(\overline{p})+X_2(\overline{q})]-\frac{\lambda}{2}(\overline{p_L}x_1+\overline{q_L}x_2),$$
    $$h_{12}=h_{21}=-\frac{l_L}{l}\langle \overline{e}_1,\nabla_H(\overline{r_L})\rangle_L-\tau{\sqrt{L}},$$
    $$h_{22}=-\frac{l^2}{l_L^2}\langle \overline{e}_2,\nabla_H(\frac{r}{l})\rangle_L+\widetilde{X_3}(\overline{r_L})+(\overline{p_L}-\overline{q_L})
    \overline{p}~\overline{q}~\overline{r_L}^2 \frac{\lambda}{2}x_2.$$
\end{thm}
\indent The Riemannian mean curvature $\mathcal{H}_L$ of $\Sigma$ is defined by
$$\mathcal{H}_L:={\rm tr}(II^L).$$
Define the curvature of a connection $\nabla$ by
\begin{equation}\label{a22}
R(X,Y)Z=\nabla_X\nabla_Y-\nabla_Y\nabla_X-\nabla_{[X,Y]}.
\end{equation}
Let
\begin{equation}
\mathcal{K}^{\Sigma,L}(\overline{e}_1,\overline{e}_2)=-\langle R^{\Sigma,L}(\overline{e}_1,\overline{e}_2)\overline{e}_1,\overline{e}_2\rangle_{\Sigma,L},~~~~\mathcal{K}^{L}(\overline{e}_1,\overline{e}_2)=-\langle R^{L}(\overline{e}_1,\overline{e}_2)\overline{e}_1,\overline{e}_2\rangle_L.
\end{equation}
By the Gauss equation, we have
\begin{equation}\label{a25}
\mathcal{K}^{\Sigma,L}(\overline{e}_1,\overline{e}_2)=\mathcal{K}^{L}(\overline{e}_1,\overline{e}_2)+{\rm det}(II^L).
\end{equation}
\begin{prop} Away from characteristic points, the horizontal mean curvature $\mathcal{H}_{\infty}$ of $\Sigma\subset N$ is given by
\begin{equation}
\mathcal{H}_{\infty}={\rm lim}_{L\rightarrow +\infty}\mathbb{}\mathcal{H}_L=X_1(\overline{p})+X_2(\overline{q})-\frac{\lambda}{2}(\overline{p}x_1+\overline{q}x_2).
\end{equation}
\end{prop}
By Lemma \ref{lemma1} and (\ref{a22}), we have
\begin{lem}\label{lemma2}
Let $N$ be the BCV spaces, then
\begin{align}
&
R^L(X_1,X_2)X_1=-\frac{\lambda^2}{4}x^2_2X_1+(-\lambda A+3\tau^2L+\frac{\lambda^2}{4}x^2_1)X_2+\frac{\lambda \tau}{2}x_2X_3
,\nonumber\\
& R^L(X_1,X_2)X_2=(\lambda-3\tau^2L)X_1,
~~~ R^L(X_1,X_2)X_3=0,~~~\nonumber\\
&R^L(X_1,X_3)X_1=-\tau^2LX_3,~~~ R^L(X_1,X_3)X_2=0,
~~~ R^L(X_1,X_3)X_3=\tau^2L^2X_1,~~~\nonumber\\
& R^L(X_2,X_3)X_1=0,~~~
R^L(X_2,X_3)X_2=-\tau^2LX_3,~~~ R^L(X_2,X_3)X_3=\tau^2L^2X_2.
\end{align}
\end{lem}
\begin{prop}\cite{WW1} Away from characteristic points, we have
\begin{equation}\label{a28}
\mathcal{K}^{\Sigma,L}(\overline{e}_1,\overline{e}_2)\rightarrow A_1+O(\frac{1}{\sqrt{L}}),~~{\rm as}~~L\rightarrow +\infty,
\end{equation}
where
\begin{equation}
A_1:=-2\tau\big\langle \overline{e}_1,\nabla_H(\frac{X_3u}{|\nabla_Hu|})\big\rangle-\frac{4\tau^2(X_3u)^2}{p^2+q^2}+\frac{\lambda\tau\overline{q}X_3u}{\sqrt{p^2+q^2}}.
\end{equation}
\end{prop}
Let $\Sigma\subset(N,g_L)$ be a Euclidean $C^2$-smooth surface and $\Sigma=\{u=0\}$. Let $d\sigma_{\Sigma,L}$ denote the surface measure on $\Sigma$ with respect to the Riemannian metric $g_L$.
Then similar to Proposition 4.2 in \cite{WW}, we have
\begin{equation}\label{a2}
{\rm lim}_{L\rightarrow +\infty}\frac{1}{\sqrt{L}}d\sigma_{\Sigma,L}=d\sigma_{\Sigma}:=\overline{p}\omega_2\wedge \omega_3-\overline{q}\omega_1\wedge \omega_3.
\end{equation}
\indent Therefore, by (\ref{a1}) and (\ref{a2}), we have
\begin{align}
{\rm lim}_{L\rightarrow +\infty}\frac{1}{\sqrt{L}}\int_{B\times F^{L}}S^{L}dvol_{B}dvol_{L}=&\int_{B\times F}\Big(f^2S^{B}
+2\widetilde{l}f\triangle_Bf+A_1+\widetilde{l}(\widetilde{l}-1)|{\rm grad}_Bf|^2_B\Big)dvol_{B}d\sigma_{\Sigma}.
\end{align}
\indent Let us recall the Kastler-Kalau-Walze type theorem. Let $M$ be a $2m$-dimensional spin manifold without boundary. Let ${\rm Wres}$ denote the Wodzicki residue and $D$ be a Dirac operator. Then
\begin{thm}(Kastler-Kalau-Walze) The following equality holds:
\begin{align}
{\rm Wres}(D^{-2m+2})=\frac{m-1}{6}\frac{2\pi ^m}{\Gamma(m)}\int_{M}Sdvol_{M},
\end{align}
where $S$ is the scalar curvature.
\end{thm}
\indent Hence, we obtain the following theorem.
\begin{thm}
Let $M=B\times_{f}F^L$ be spin and $D$ be the associated Dirac operator. Suppose that the characteristic set $C(F)$ satisfies $\mathcal{H}^1(C(F))=0$ where $\mathcal{H}^1(C(F))$ denotes the Euclidean $1$-dimensional Hausdorff measure of $C(F)$ and that $||\nabla_Hu||_H^{-1}$ is locally summable with respect to the Euclidean $2$-dimensional Hausdorff measure near the characteristic set $C(F)$ (for related definition, see \cite{BTV}), then
\begin{align}
{\rm lim}_{L\rightarrow +\infty}\frac{1}{\sqrt{L}}{\rm Wres}(D^{-2m+2})=&\frac{m-1}{6}\frac{2\pi ^m}{\Gamma(m)}\int_{B\times F}\Big(f^2S^{B}
+2\widetilde{l}f\triangle_Bf+A_1+\widetilde{l}(\widetilde{l}-1)|{\rm grad}_Bf|^2_B\Big)dvol_{B}d\sigma_{\Sigma}.
\end{align}
\end{thm}

\section{ The sub-Riemannian Dabrowski-Sitarz-Zalecki type theorem for the twisted BCV spaces }
In this section, we will compute the sub-Riemannian limit of the Einstein functional. The following is discussed in three cases:\\

\noindent {\bf Case (a)} By Proposition \ref{prop1}, when $X,Y\in \Gamma(TB)$,  we have
\begin{align}
{\rm Ric}^{L} (X,Y)-\frac{S^L}{2}g^{L}(X,Y)=&{\rm Ric}^B(X,Y)+\frac{\widetilde{l}}{f}H_B^{f}(X,Y)-\frac{S^L}{2}g^{B}(X,Y)\nonumber\\
=&{\rm Ric}^B(X,Y)+\frac{\widetilde{l}}{f}H_B^{f}(X,Y)-\Big(\frac{S^B}{2}+\frac{\widetilde{l}}{f}\triangle_Bf+\frac{S^{F^L}}{2f^2}\nonumber\\
&+\widetilde{l}(\widetilde{l}-1)\frac{|{\rm grad}_Bf|^2_B}{2f^2}\Big)g^{B}(X,Y).
\end{align}
In this case, we obtain
\begin{align}\label{a29}
&{\rm lim}_{L\rightarrow +\infty}\frac{1}{\sqrt{L}}\int_{B\times F^{L}}\Big({\rm Ric}^{L} (X,Y)-\frac{S^L}{2}g^{L}(X,Y)\Big)dvol_{B}dvol_{L}\nonumber\\
=&\int_{B\times F}\bigg[f^2{\rm Ric}^B(X,Y)+\widetilde{l}fH_B^{f}(X,Y)-\Big(f^2\frac{S^{B}}{2}+\widetilde{l}f\triangle_Bf+\frac{A_1}{2}+\frac{\widetilde{l}(\widetilde{l}-1)}{2}|{\rm grad}_Bf|^2_B\Big)g^{B}(X,Y)\bigg]dvol_{B}d\sigma_{\Sigma}.
\end{align}
\noindent {\bf Case (b)} By Proposition \ref{prop1}, when $X\in \Gamma(TB), Y\in \Gamma(TF)$, we have
\begin{align}
{\rm Ric}^{L} (X,Y)-\frac{S^L}{2}g^{L}(X,Y)=(\widetilde{l}-1)[YX({\rm ln}f)].
\end{align}
In this case, we obtain
\begin{align}\label{a30}
&{\rm lim}_{L\rightarrow +\infty}\frac{1}{\sqrt{L}}\int_{B\times F^{L}}\Big({\rm Ric}^{L} (X,Y)-\frac{S^L}{2}g^{L}(X,Y)\Big)dvol_{B}dvol_{L}\nonumber\\
=&\int_{B\times F}(\widetilde{l}-1)f^2[YX({\rm ln}f)]dvol_{B}d\sigma_{\Sigma}.
\end{align}
\noindent {\bf Case (c)} By Proposition \ref{prop1}, when $X,Y\in \Gamma(TF)$, we have
\begin{align}\label{a7}
{\rm Ric}^{L} (X,Y)-\frac{S^L}{2}g^{L}(X,Y)=&{\rm Ric}^{F^L}
(X,Y)+\left[\frac{\triangle_Bf}{f}+(\widetilde{l}-1)\frac{|{\rm
grad}_Bf|^2_B}{f^2}\right]f^2g^L(X,Y)\nonumber\\
&-\frac{S^L}{2}f^2g^{L}(X,Y).
\end{align}
Set $$X=\sum_{j=1}^3a_jX_j, \ \ Y=\sum_{j=1}^3b_jX_j,$$ then
\begin{align}
X=&<X,\overline{e}_1^L>\overline{e}_1^L+<X,\overline{e}_2^L>\overline{e}_2^L\nonumber\\
=&\Big<\sum_{j=1}^3a_jX_j,\overline{q}X_1-\overline{p}X_2\Big>_{L}\overline{e}_1^L+\Big<\sum_{j=1}^3a_jX_j,\overline{r_L}~~\overline{p}X_1+\overline{r_L}~~ \overline{q}X_2-\frac{l}{l_L}\widetilde{X_3}\Big>_{L}\overline{e}_2^L\nonumber\\
=&\big(a_1\overline{q}-a_2\overline{p}\big)\overline{e}_1^L+\big(a_1\overline{r_L}~\overline{p}+a_2\overline{r_L}~\overline{q}-a_3\frac{l}{l_L}L^{\frac{1}{2}}\big)\overline{e}_2^L\nonumber\\
:=&\Phi_1\overline{e}_1^L+\Phi_2\overline{e}_2^L.
\end{align}
Similarly, we have
\begin{align}
Y=&\big(b_1\overline{q}-b_2\overline{p}\big)\overline{e}_1^L+\big(b_1\overline{r_L}~\overline{p}+b_2\overline{r_L}~\overline{q}-b_3\frac{l}{l_L}L^{\frac{1}{2}}\big)\overline{e}_2^L\nonumber\\
:=&\Psi_1\overline{e}_1^L+\Psi_2\overline{e}_2^L.
\end{align}
So
\begin{align}\label{a5}
g^{L}(X,Y)=&\big(a_1\overline{q}-a_2\overline{p}\big)\big(b_1\overline{q}-b_2\overline{p}\big)\nonumber\\
&+\big(a_1\overline{r_L}~\overline{p}+a_2\overline{r_L}~\overline{q}-a_3\frac{l}{l_L}L^{\frac{1}{2}}\big)\big(b_1\overline{r_L}~\overline{p}+b_2\overline{r_L}~\overline{q}-b_3\frac{l}{l_L}L^{\frac{1}{2}}\big)\nonumber\\
=&\Phi_1\Psi_1+\Phi_2\Psi_2,
\end{align}
and
\begin{align}
{\rm Ric}^{F^L}(X,Y)=&{\rm Ric}^{L}\big(\Phi_1\overline{e}_1^L+\Phi_2\overline{e}_2^L,\Psi_1\overline{e}_1^L+\Psi_2\overline{e}_2^L\big)\nonumber\\
=&\Phi_1\Psi_1{\rm Ric}^{L}(\overline{e}_1^L,\overline{e}_1^L)+\Phi_1\Psi_2{\rm Ric}^{L}(\overline{e}_1^L,\overline{e}_2^L)\nonumber\\
&+\Phi_2\Psi_1{\rm Ric}^{L}(\overline{e}_2^L,\overline{e}_1^L)+\Phi_2\Psi_2{\rm Ric}^{L}(\overline{e}_2^L,\overline{e}_2^L).
\end{align}
By definition  $${\rm Ric}^{L}(X,Y)=\sum_{j=1}^2R^L(\overline{e}_j^L,X,\overline{e}_j^L,Y),$$ we have
\begin{align}\label{a3}
{\rm Ric}^{L}(\overline{e}_1^L,\overline{e}_1^L)=&R^L(\overline{e}_2^L,\overline{e}_1^L,\overline{e}_2^L,\overline{e}_1^L)\nonumber\\
=&R^L(\overline{e}_1^L,\overline{e}_2^L,\overline{e}_1^L,\overline{e}_2^L)\nonumber\\
=&\mathcal{K}^{\Sigma,L}(\overline{e}_1,\overline{e}_2).
\end{align}
Likewise, we have
\begin{align}\label{a4}
{\rm Ric}^{L}(\overline{e}_2^L,\overline{e}_2^L)=\mathcal{K}^{\Sigma,L}(\overline{e}_1,\overline{e}_2).
\end{align}
By (\ref{a3}), (\ref{a4}) and $${\rm Ric}^{L}(\overline{e}_1^L,\overline{e}_2^L)={\rm Ric}^{L}(\overline{e}_2^L,\overline{e}_1^L)=0,$$ we get
\begin{align}\label{a6}
{\rm Ric}^{F^L}(X,Y)=&\Phi_1\Psi_1\mathcal{K}^{\Sigma,L}(\overline{e}_1,\overline{e}_2)+\Phi_2\Psi_2\mathcal{K}^{\Sigma,L}(\overline{e}_1,\overline{e}_2)\nonumber\\
=&\big(\Phi_1\Psi_1+\Phi_2\Psi_2\big)\mathcal{K}^{\Sigma,L}(\overline{e}_1,\overline{e}_2)\nonumber\\
=&g^{L}(X,Y)\mathcal{K}^{\Sigma,L}(\overline{e}_1,\overline{e}_2).
\end{align}
Substituting the results from (\ref{a5}) and (\ref{a6}) into (\ref{a7}), this yields
\begin{align}
{\rm Ric}^{L} (X,Y)-\frac{S^L}{2}g^{L}(X,Y)=&{\rm Ric}^{F^L}
(X,Y)+\left[\frac{\triangle_Bf}{f}+(\widetilde{l}-1)\frac{|{\rm
grad}_Bf|^2_B}{f^2}\right]f^2g^L(X,Y)\nonumber\\
&-\frac{S^L}{2}f^2g^{L}(X,Y)\nonumber\\
=&\Big(\mathcal{K}^{\Sigma,L}(\overline{e}_1,\overline{e}_2)+(1-\widetilde{l})f\triangle_Bf-f^2\frac{S^B}{2}-\frac{S^{F^L}}{2}\nonumber\\
&+\frac{(2-\widetilde{l})(\widetilde{l}-1)}{2}|{\rm grad}_Bf|^2_B\Big)g^{L}(X,Y).
\end{align}
In this case, we obtain
\begin{align}\label{a31}
&{\rm lim}_{L\rightarrow +\infty}\frac{1}{L}\frac{1}{\sqrt{L}}\int_{B\times F^{L}}\Big({\rm Ric}^{L} (X,Y)-\frac{S^L}{2}g^{L}(X,Y)\Big)dvol_{B}dvol_{L}\nonumber\\
=&\int_{B\times F}\Big(f^2\frac{A_1}{2}+(1-\widetilde{l})f^3\triangle_Bf-f^4\frac{S^B}{2}+\frac{(2-\widetilde{l})(\widetilde{l}-1)}{2}f^2|{\rm grad}_Bf|^2_B\Big)a_3b_3dvol_{B}d\sigma_{\Sigma}.
\end{align}
\indent Let us recall the Dabrowski-Sitarz-Zalecki type theorem. Let $M$ be a $2m$-dimensional spin manifold without boundary. Let ${\rm Wres}$ denote the Wodzicki residue and $D$ be a Dirac operator and $\nabla^{S(TM)}$ be the spin connection and $X,Y$ be two smooth vector fields. Then
\begin{thm}(Dabrowski-Sitarz-Zalecki) The following equalities hold:
\begin{align}
&{\rm Wres}(\nabla_X^{S(TM)}\nabla_Y^{S(TM)}D^{-2m})=\frac{2^m}{6}\frac{2\pi ^m}{\Gamma(m)}\int_{M}\Big({\rm Ric} (X,Y)-\frac{1}{2}Sg(X,Y)\Big)dvol_{M},\nonumber\\
&{\rm Wres}\Big(c(X)\big(Dc(Y)+c(Y)D\big)D^{-2m+1}\Big)=\frac{2^m}{6}\frac{2\pi ^m}{\Gamma(m)}\int_{M}\Big({\rm Ric} (X,Y)-\frac{1}{2}Sg(X,Y)\Big)dvol_{M}.
\end{align}
\end{thm}
By (\ref{a29}), we have
\begin{thm}
Let $M=B\times_{f}F^L$ be spin and $D$ be the associated Dirac operator. Suppose that the characteristic set $C(F)$ satisfies $\mathcal{H}^1(C(F))=0$ and that $||\nabla_Hu||_H^{-1}$ is locally summable with respect to the Euclidean $2$-dimensional Hausdorff measure near the characteristic set $C(F)$, then the following equalities hold for $X,Y\in \Gamma(TB)$:
\begin{align}
&{\rm lim}_{L\rightarrow +\infty}\frac{1}{\sqrt{L}}{\rm Wres}(\nabla_X^{S(TM)}\nabla_Y^{S(TM)}D^{-2m})\nonumber\\
=&\frac{2^m}{6}\frac{2\pi ^m}{\Gamma(m)}\int_{B\times F}\bigg[f^2{\rm Ric}^B(X,Y)+\widetilde{l}fH_B^{f}(X,Y)-\Big(f^2\frac{S^{B}}{2}+\widetilde{l}f\triangle_Bf+\frac{A_1}{2}\nonumber\\
&+\frac{\widetilde{l}(\widetilde{l}-1)}{2}|{\rm grad}_Bf|^2_B\Big)g^{B}(X,Y)\bigg]dvol_{B}d\sigma_{\Sigma},\nonumber\\
&{\rm lim}_{L\rightarrow +\infty}\frac{1}{\sqrt{L}}{\rm Wres}\Big(c(X)\big(Dc(Y)+c(Y)D\big)D^{-2m+1}\Big)\nonumber\\
=&\frac{2^m}{6}\frac{2\pi ^m}{\Gamma(m)}\int_{B\times F}\bigg[f^2{\rm Ric}^B(X,Y)+\widetilde{l}fH_B^{f}(X,Y)-\Big(f^2\frac{S^{B}}{2}+\widetilde{l}f\triangle_Bf+\frac{A_1}{2}\nonumber\\
&+\frac{\widetilde{l}(\widetilde{l}-1)}{2}|{\rm grad}_Bf|^2_B\Big)g^{B}(X,Y)\bigg]dvol_{B}d\sigma_{\Sigma}.
\end{align}
\end{thm}
By (\ref{a30}), we have
\begin{thm}
Let $M=B\times_{f}F^L$ be spin and $D$ be the associated Dirac operator. Suppose that the characteristic set $C(F)$ satisfies $\mathcal{H}^1(C(F))=0$ and that $||\nabla_Hu||_H^{-1}$ is locally summable with respect to the Euclidean $2$-dimensional Hausdorff measure near the characteristic set $C(F)$, then the following equalities hold for $X\in \Gamma(TB)$ and $Y\in \Gamma(TF)$:
\begin{align}
&{\rm lim}_{L\rightarrow +\infty}\frac{1}{\sqrt{L}}{\rm Wres}(\nabla_X^{S(TM)}\nabla_Y^{S(TM)}D^{-2m})=\frac{2^m}{6}\frac{2\pi ^m}{\Gamma(m)}\int_{B\times F}(\widetilde{l}-1)f^2[YX({\rm ln}f)]dvol_{B}d\sigma_{\Sigma},\nonumber\\
&{\rm lim}_{L\rightarrow +\infty}\frac{1}{\sqrt{L}}{\rm Wres}\Big(c(X)\big(Dc(Y)+c(Y)D\big)D^{-2m+1}\Big)=\frac{2^m}{6}\frac{2\pi ^m}{\Gamma(m)}\int_{B\times F}(\widetilde{l}-1)f^2[YX({\rm ln}f)]dvol_{B}d\sigma_{\Sigma}.
\end{align}
\end{thm}
By (\ref{a31}), we have
\begin{thm}
Let $M=B\times_{f}F^L$ be spin and $D$ be the associated Dirac operator. Suppose that the characteristic set $C(F)$ satisfies $\mathcal{H}^1(C(F))=0$ and that $||\nabla_Hu||_H^{-1}$ is locally summable with respect to the Euclidean $2$-dimensional Hausdorff measure near the characteristic set $C(F)$, then the following equalities hold for $X,Y\in \Gamma(TF)$:
\begin{align}
&{\rm lim}_{L\rightarrow +\infty}\frac{1}{(\sqrt{L})^3}{\rm Wres}(\nabla_X^{S(TM)}\nabla_Y^{S(TM)}D^{-2m})\nonumber\\
=&\frac{2^m}{6}\frac{2\pi ^m}{\Gamma(m)}\int_{B\times F}\Big(f^2\frac{A_1}{2}+(1-\widetilde{l})f^3\triangle_Bf-f^4\frac{S^B}{2}+\frac{(2-\widetilde{l})(\widetilde{l}-1)}{2}f^2|{\rm grad}_Bf|^2_B\Big)a_3b_3dvol_{B}d\sigma_{\Sigma},\nonumber\\
&{\rm lim}_{L\rightarrow +\infty}\frac{1}{(\sqrt{L})^3}{\rm Wres}\Big(c(X)\big(Dc(Y)+c(Y)D\big)D^{-2m+1}\Big)\nonumber\\
=&\frac{2^m}{6}\frac{2\pi ^m}{\Gamma(m)}\int_{B\times F}\Big(f^2\frac{A_1}{2}+(1-\widetilde{l})f^3\triangle_Bf-f^4\frac{S^B}{2}+\frac{(2-\widetilde{l})(\widetilde{l}-1)}{2}f^2|{\rm grad}_Bf|^2_B\Big)a_3b_3dvol_{B}d\sigma_{\Sigma}.
\end{align}
\end{thm}

\section{ The sub-Riemannian limit of the Connes conformal invariants in the twisted BCV spaces }
In this section, we will compute the sub-Riemannian limit of the Connes conformal invariants on 4-dimensional Riemannian manifolds without boundary.\\
\indent Using Eq.(2.33) in \cite{Co1}, we get the following theorem on 4-dimensional manifolds:
\begin{thm}\cite{Co1}
There exists a universal trilinear form in the curvature
$r$ and the covectors $df_1,df_2$ such that, in full generality, one has:
\begin{align}\label{aaaa}
\Omega_4(f_1,f_2)_g&=\Big[\frac{1}{3}r\langle df_1,df_2\rangle+\Delta\langle df_1,df_2\rangle +\langle \nabla df_1,\nabla df_2\rangle -\frac{1}{2}\Delta f_1\Delta f_2\Big]Vol_{M,g},
\end{align}
where $\Delta$ is the Laplacian, $\nabla$ denotes the covariant derivative and $\Omega_4(f_1,f_2)_g$ is a conformal invariant. Then $${\rm Wres}\big(f_0[\widetilde{F},f_1][\widetilde{F},f_2]\big)=\int_Mf_0\Omega_4(f_1,f_2)_g,$$
where the definition of $\widetilde{F}$ (see \cite{Co1}) and ${\rm Wres}\big(f_0[\widetilde{F},f_1][\widetilde{F},f_2]\big)$ is a Hochschild 2-cocyle on $C^{\infty}(M)$.
\end{thm}
On $n$-dimensonal manifolds, let $\{e_1^*, \cdot\cdot\cdot, e_m^*\}$ be an orthonormal basis on $T^*B$ and $\{f\overline{e}_1^*, \cdot\cdot\cdot, f\overline{e}_n^*\}$ be an  orthonormal basis on $T^*F$ with respect to $g_L$, then we know
\begin{align}
Vol_{M,g}=&e_1^*\wedge\cdot\cdot\cdot\wedge e_m^*\wedge f\overline{e}_1^*\wedge\cdot\cdot\cdot\wedge f\overline{e}_n^*\nonumber\\
=&f^ndvol_Bdvol_F.
\end{align}

Firstly, for $n=2$, we will compute $$\frac{1}{3}r\langle df_1,df_2\rangle ,~~\Delta\langle df_1,df_2\rangle ,~~\langle \nabla df_1,\nabla df_2\rangle ,~~-\frac{1}{2}\Delta f_1\Delta f_2$$ respectively.\\

\noindent (1) For $df_1\in T^*M$, we have
\begin{align}
df_1=\sum_{j=1}^me_j(f_1)e_j^*+\sum_{k=1}^n\overline{e}_k(f_1)\overline{e}_k^*.
\end{align}
So
\begin{align}
\langle df_1,df_2\rangle =&\sum_{j=1}^me_j(f_1)e_j(f_2)+\sum_{k=1}^n\frac{1}{f^2}\overline{e}_k(f_1)\overline{e}_k(f_2)\nonumber\\
=&\langle d_Bf_1,d_Bf_2\rangle _B+\frac{1}{f^2}\langle d_Ff_1,d_Ff_2\rangle _F.
\end{align}
As a result
\begin{align}\label{a9}
\frac{1}{3}r\langle df_1,df_2\rangle =&\frac{1}{3}S\Big(\langle d_Bf_1,d_Bf_2\rangle _B+\frac{1}{f^2}\langle d_Ff_1,d_Ff_2\rangle _F\Big)\nonumber\\
=&\frac{1}{3}\Big(S^B+2\frac{\widetilde{l}}{f}\triangle_Bf+\frac{S^{F}}{f^2}+\widetilde{l}(\widetilde{l}-1)\frac{|{\rm grad}_Bf|^2_B}{f^2}\Big)\langle d_Bf_1,d_Bf_2\rangle _B\nonumber\\
&+\frac{1}{3}\Big(\frac{S^B}{f^2}+2\frac{\widetilde{l}}{f^3}\triangle_Bf+\frac{S^{F}}{f^4}+\widetilde{l}(\widetilde{l}-1)\frac{|{\rm grad}_Bf|^2_B}{f^4}\Big)\langle d_Ff_1,d_Ff_2\rangle _F\nonumber\\
:=&c_1.
\end{align}
\noindent (2) By Proposition \ref{prop2}, we have
\begin{align}\label{a8}
\Delta=&-\sum_{j=1}^{m+n}e_je_j+\sum_{j=1}^{m+n}\nabla_{e_j}^Le_j\nonumber\\
=&-\sum_{j=1}^{m}e_je_j-\sum_{k=1}^{n}\frac{1}{f}\overline{e}_k\frac{1}{f}\overline{e}_k+\sum_{j=1}^{m}\nabla_{e_j}^Be_j+\sum_{k=1}^{n}\nabla_{\frac{1}{f}\overline{e}_k}\frac{1}{f}\overline{e}_k\nonumber\\
=&-\sum_{j=1}^{m}e_je_j+\sum_{j=1}^{m}\nabla_{e_j}^Be_j-\sum_{k=1}^{n}\frac{1}{f^2}\overline{e}_k\overline{e}_k-\sum_{k=1}^{n}\frac{1}{f}\overline{e}_k(\frac{1}{f})\overline{e}_k+\frac{1}{f}\sum_{k=1}^{n}\nabla_{\overline{e}_k}\frac{1}{f}\overline{e}_k\nonumber\\
=&\Delta_B-\frac{1}{f^2}\sum_{k=1}^{n}\overline{e}_k\overline{e}_k+\frac{1}{f^2}\sum_{k=1}^{n}\nabla_{\overline{e}_k}\overline{e}_k\nonumber\\
=&\Delta_B-\frac{1}{f^2}\sum_{k=1}^{n}\overline{e}_k\overline{e}_k+\frac{1}{f^2}\sum_{k=1}^{n}\Big(2\overline{e}_k({\rm ln}f)\overline{e}_k-\frac{g_{F}(\overline{e}_k,\overline{e}_k)}{f}{\rm grad}_{F}f-fg_{F}(\overline{e}_k,\overline{e}_k){\rm grad}_{B}f+\nabla^{F}_{\overline{e}_k}\overline{e}_k\Big)\nonumber\\
=&\Delta_B+\frac{1}{f^2}\Delta_F+\frac{2-n}{f^3}{\rm grad}_{F}f-\frac{n}{f}{\rm grad}_{B}f.
\end{align}
So
\begin{align}\label{a10}
\Delta\langle df_1,df_2\rangle
=&\Big(\Delta_B+\frac{1}{f^2}\Delta_F+\frac{2-n}{f^3}{\rm grad}_{F}f-\frac{n}{f}{\rm grad}_{B}f\Big)\Big(\langle d_Bf_1,d_Bf_2\rangle _B+\frac{1}{f^2}\langle d_Ff_1,d_Ff_2\rangle _F\Big)\nonumber\\
=&\Big(\Delta_B+\frac{1}{f^2}\Delta_F+\frac{2-n}{f^3}{\rm grad}_{F}f-\frac{n}{f}{\rm grad}_{B}f\Big)\langle d_Bf_1,d_Bf_2\rangle _B+\Delta_B(\frac{1}{f^2})\langle d_Ff_1,d_Ff_2\rangle _F\nonumber\\
&+\frac{1}{f^2}\Delta_B\big(\langle d_Ff_1,d_Ff_2\rangle _F\big)-2\Big\langle {\rm grad}_{B}(\frac{1}{f^2}),{\rm grad}_{B}\big(\langle d_Ff_1,d_Ff_2\rangle _F\big)\Big\rangle _B+\frac{1}{f^2}\Big[\Delta_F(\frac{1}{f^2})\nonumber\\
&\times\langle d_Ff_1,d_Ff_2\rangle _F+\frac{1}{f^2}\Delta_F\big(\langle d_Ff_1,d_Ff_2\rangle _F\big)-2\Big\langle {\rm grad}_{F}(\frac{1}{f^2}),{\rm grad}_{F}\big(\langle d_Ff_1,d_Ff_2\rangle _F\big)\Big\rangle _F\Big]\nonumber\\
&+\frac{2-n}{f^3}{\rm grad}_{F}f(\frac{1}{f^2})\langle d_Ff_1,d_Ff_2\rangle _F+\frac{2-n}{f^5}{\rm grad}_{F}f\big(\langle d_Ff_1,d_Ff_2\rangle _F\big)-\frac{n}{f}{\rm grad}_{B}f(\frac{1}{f^2})\nonumber\\
&\times\langle d_Ff_1,d_Ff_2\rangle _F-\frac{n}{f^3}{\rm grad}_{B}f\big(\langle d_Ff_1,d_Ff_2\rangle _F\big)\nonumber\\
:=&c_2.
\end{align}
\noindent (3) By Proposition \ref{prop3}, we have
\begin{align}
\nabla df_1=\sum_{j=1}^me_j^* \otimes \nabla_{e_j}^*(df_1)+\sum_{k=1}^n\overline{e}_k^* \otimes \nabla_{\overline{e}_k}^*(df_1).
\end{align}
By further calculation, we have
\begin{align}
\nabla_{e_j}^*(df_1)=&\nabla_{e_j}^*\Big(\sum_{\alpha=1}^me_\alpha(f_1)e_\alpha^*+\sum_{\beta=1}^n\overline{e}_\beta(f_1)\overline{e}_\beta^*\Big)\nonumber\\
=&\sum_{\alpha=1}^me_j\big(e_\alpha(f_1)\big)e_\alpha^*+\sum_{\alpha=1}^me_\alpha(f_1)\nabla_{e_j}^*e_\alpha^*
+\sum_{\beta=1}^ne_j\big(\overline{e}_\beta(f_1)\big)\overline{e}_\beta^*+\sum_{\beta=1}^n\overline{e}_\beta(f_1)\nabla_{e_j}^*\overline{e}_\beta^*\nonumber\\
=&\sum_{\alpha=1}^me_j\big(e_\alpha(f_1)\big)e_\alpha^*+\sum_{\alpha=1}^me_\alpha(f_1)\nabla_{e_j}^{B,*}e_\alpha^*
+\sum_{\beta=1}^ne_j\big(\overline{e}_\beta(f_1)\big)\overline{e}_\beta^*-\sum_{\beta=1}^n\overline{e}_\beta(f_1)\frac{e_j(f)}{f}\overline{e}_\beta^*,
\end{align}
and
\begin{align}
\nabla_{\overline{e}_k}^*(df_1)=&\nabla_{\overline{e}_k}^*\Big(\sum_{\alpha=1}^me_\alpha(f_1)e_\alpha^*+\sum_{\beta=1}^n\overline{e}_\beta(f_1)\overline{e}_\beta^*\Big)\nonumber\\
=&\sum_{\alpha=1}^m\overline{e}_k\big(e_\alpha(f_1)\big)e_\alpha^*+\sum_{\alpha=1}^me_\alpha(f_1)\nabla_{\overline{e}_k}^*e_\alpha^*
+\sum_{\beta=1}^n\overline{e}_k\big(\overline{e}_\beta(f_1)\big)\overline{e}_\beta^*+\sum_{\beta=1}^n\overline{e}_\beta(f_1)\nabla_{\overline{e}_k}^*\overline{e}_\beta^*\nonumber\\
=&\sum_{\alpha=1}^m\overline{e}_k\big(e_\alpha(f_1)\big)e_\alpha^*+f\sum_{\alpha=1}^me_\alpha(f_1)e_\alpha(f)\overline{e}_k^*
+\sum_{\beta=1}^n\overline{e}_k\big(\overline{e}_\beta(f_1)\big)\overline{e}_\beta^*-\sum_{\beta=1}^n\overline{e}_\beta(f_1)\overline{e}_k(lnf)\overline{e}_\beta^*\nonumber\\
&-\sum_{\beta=1}^n\overline{e}_\beta(f_1)\delta_k^{\beta}\big({\rm grad}_{F}(lnf)\big)^*-\frac{1}{f}\sum_{\beta=1}^n\overline{e}_\beta(f_1)\delta_k^{\beta}\big({\rm grad}_{B}f\big)^*+\frac{1}{f}\sum_{\beta=1}^n\overline{e}_\beta(f_1)\overline{e}_\beta(f)\overline{e}_k^*\nonumber\\
&+\sum_{\beta=1}^n\overline{e}_\beta(f_1)\nabla_{\overline{e}_k}^{F,*}\overline{e}_\beta^*.
\end{align}
These lead to
\begin{align}
\nabla df_1=&\sum_{j,\alpha=1}^me_j\big(e_\alpha(f_1)\big)e_j^*\otimes e_\alpha^*+\sum_{j,\alpha=1}^me_\alpha(f_1)e_j^*\otimes \nabla_{e_j}^{B,*}e_\alpha^*+\sum_{j=1}^m\sum_{\beta=1}^ne_j\big(\overline{e}_\beta(f_1)\big)e_j^*\otimes \overline{e}_\beta^*\nonumber\\
&-\frac{1}{f}\sum_{j=1}^m\sum_{\beta=1}^n\overline{e}_\beta(f_1)e_j(f)e_j^*\otimes \overline{e}_\beta^*+\sum_{\alpha=1}^m\sum_{k=1}^n\overline{e}_k\big(e_\alpha(f_1)\big)\overline{e}_k^*\otimes e_\alpha^*\nonumber\\
&+f\sum_{\alpha=1}^m\sum_{k=1}^ne_\alpha(f_1)e_\alpha(f)\overline{e}_k^*\otimes\overline{e}_k^*+\sum_{k,\beta=1}^n\overline{e}_k\big(\overline{e}_\beta(f_1)\big)\overline{e}_k^*\otimes\overline{e}_\beta^*\nonumber\\
&-\big({\rm grad}_{F}(lnf)\big)^*\otimes\big({\rm grad}_{F}f_1\big)^*-\big({\rm grad}_{F}f_1\big)^*\otimes\big({\rm grad}_{F}(lnf)\big)^*\nonumber\\
&-\frac{1}{f}\big({\rm grad}_{F}f_1\big)^*\otimes\big({\rm grad}_{B}f\big)^*
+\frac{1}{f}\sum_{k,\beta=1}^n\overline{e}_\beta(f_1)\overline{e}_\beta(f)\overline{e}_k^*\otimes\overline{e}_k^*\nonumber\\
&+\sum_{k,\beta=1}^n\overline{e}_\beta(f_1)\overline{e}_k^*\otimes\nabla_{\overline{e}_k}^{F,*}\overline{e}_\beta^*.
\end{align}
Moreover, by $$\big\langle e_j^*\otimes e_k^*,e_\alpha^*\otimes e_\beta^*\big\rangle_f =\delta_j^\alpha\delta_k^\beta$$ and $$\big\langle \overline{e}_j^*\otimes \overline{e}_k^*,\overline{e}_\alpha^*\otimes \overline{e}_\beta^*\big\rangle_f =\frac{1}{f^4}\delta_j^\alpha\delta_k^\beta,$$ we obtain
\begin{align}\label{a11}
&\langle \nabla df_1,\nabla df_2\rangle\nonumber\\
=&\sum_{j,\alpha=1}^me_j\big(e_\alpha(f_1)\big)e_{j}\big(e_{\alpha}(f_2)\big)+\sum_{j,\alpha,\alpha'=1}^me_j\big(e_\alpha(f_1)\big)e_{\alpha'}(f_2)\big\langle e_\alpha^*,\nabla_{e_{j}}^{B,*}e_{\alpha'}^*\big\rangle\nonumber\\
&+\sum_{j,\alpha,\alpha'=1}^me_\alpha(f_1)e_{j}\big(e_{\alpha'}(f_2)\big)\big\langle \nabla_{e_{j}}^{B,*}e_{\alpha}^*,e_{\alpha'}^*\big\rangle+\sum_{j,\alpha,\alpha'=1}^me_\alpha(f_1)e_{\alpha'}(f_2)\big\langle \nabla_{e_{j}}^{B,*}e_{\alpha}^*,\nabla_{e_{j}}^{B,*}e_{\alpha'}^*\big\rangle\nonumber\\
&+\frac{1}{f^2}\sum_{j=1}^m\sum_{\beta=1}^ne_j\big(\overline{e}_\beta(f_1)\big)e_{j}\big(\overline{e}_{\beta}(f_2)\big)-\frac{1}{f^3}\sum_{j=1}^m\sum_{\beta=1}^ne_j\big(\overline{e}_\beta(f_1)\big)\overline{e}_{\beta}(f_2)e_{j}(f)\nonumber\\
&-\frac{1}{f^3}\sum_{j=1}^m\sum_{\beta=1}^n\overline{e}_\beta(f_1)e_{j}(f)e_{j}\big(\overline{e}_{\beta}(f_2)\big)+\frac{1}{f^4}\sum_{j=1}^m\sum_{\beta=1}^n\overline{e}_\beta(f_1)e_{j}(f)\overline{e}_{\beta}(f_2)e_{j}(f)\nonumber\\
&+\frac{1}{f^2}\sum_{\alpha=1}^m\sum_{k=1}^n\overline{e}_k\big(e_\alpha(f_1)\big)\overline{e}_{k}\big(e_{\alpha}(f_2)\big)-\frac{1}{f^3}\sum_{\alpha=1}^m\sum_{k=1}^n\overline{e}_k\big(e_\alpha(f_1)\big)\overline{e}_k(f_2)e_{\alpha}(f)\nonumber\\
&+\frac{1}{f^2}\sum_{\alpha,\alpha'=1}^me_\alpha(f_1)e_{\alpha'}(f_2)e_\alpha(f)e_{\alpha'}(f)+\frac{1}{f^3}\sum_{\alpha=1}^m\sum_{k=1}^ne_\alpha(f_1)\overline{e}_{k}\big(\overline{e}_k(f_2)\big)e_{\alpha}(f)\nonumber\\
&-\frac{1}{f^3}\sum_{\alpha=1}^m\sum_{k=1}^ne_\alpha(f_1)\overline{e}_k(f_2)\overline{e}_{k}(lnf)e_\alpha(f)-\frac{1}{f^3}\sum_{\alpha=1}^m\sum_{k=1}^ne_\alpha(f_1)\overline{e}_k(f_2)\overline{e}_{k}(lnf)e_\alpha(f)\nonumber\\
&+\frac{1}{f^4}\sum_{\alpha=1}^m\sum_{\beta'=1}^ne_\alpha(f_1)\overline{e}_{\beta'}(f_2)e_\alpha(f)\overline{e}_{\beta'}(f)+\frac{1}{f^3}\sum_{\alpha=1}^m\sum_{k,\beta'=1}^ne_\alpha(f_1)\overline{e}_{\beta'}(f_2)e_\alpha(f)\big\langle \overline{e}_k^*,\nabla_{\overline{e}_{k}}^{F,*}\overline{e}_{\beta'}^*\big\rangle_F\nonumber\\
&+\frac{1}{f^3}\sum_{\alpha'=1}^m\sum_{k=1}^n\overline{e}_k\big(\overline{e}_k(f_1)\big)e_{\alpha'}(f_2)e_{\alpha'}(f)+\frac{1}{f^4}\sum_{k,\beta=1}^n\overline{e}_k\big(\overline{e}_\beta(f_1)\big)\overline{e}_{k}\big(\overline{e}_{\beta}(f_2)\big)\nonumber\\
&-\frac{1}{f^4}\sum_{k,\beta=1}^n\overline{e}_k\big(\overline{e}_\beta(f_1)\big)\overline{e}_{\beta}(f_2)\overline{e}_{k}(lnf)-\frac{1}{f^4}\sum_{k,\beta=1}^n\overline{e}_k\big(\overline{e}_\beta(f_1)\big)\overline{e}_k(f_2)\overline{e}_{\beta}(lnf)\nonumber\\
&+\frac{1}{f^5}\sum_{k,\beta'=1}^n\overline{e}_k\big(\overline{e}_k(f_1)\big)\overline{e}_{\beta'}(f_2)\overline{e}_{\beta'}(f)+\frac{1}{f^4}\sum_{k,\beta,\beta'=1}^n\overline{e}_k\big(\overline{e}_\beta(f_1)\big)\overline{e}_{\beta'}(f_2)\big\langle \overline{e}_k^*,\nabla_{\overline{e}_{k}}^{F,*}\overline{e}_{\beta'}^*\big\rangle_F\nonumber\\
&-\frac{1}{f^3}\sum_{\alpha'=1}^m\sum_{k=1}^n\overline{e}_k(f_1)e_{\alpha'}(f_2)\overline{e}_{k}(lnf)e_{\alpha'}(f)-\frac{1}{f^4}\sum_{k,\beta=1}^n\overline{e}_\beta(f_1)\overline{e}_{k}\big(\overline{e}_{\beta}(f_2)\big)\overline{e}_{k}(lnf)\nonumber\\
&+\frac{1}{f^4}\sum_{k,\beta=1}^n\overline{e}_\beta(f_1)\overline{e}_{\beta}(f_2)\overline{e}_{k}(lnf)\overline{e}_{k}(lnf)+\frac{1}{f^4}\sum_{k,\beta=1}^n\overline{e}_\beta(f_1)\overline{e}_k(f_2)\overline{e}_{k}(lnf)\overline{e}_{\beta}(lnf)\nonumber\\
&-\frac{1}{f^5}\sum_{k,\beta'=1}^n\overline{e}_k(f_1)\overline{e}_{\beta'}(f_2)\overline{e}_{k}(lnf)\overline{e}_{\beta'}(f)-\frac{1}{f^4}\sum_{k,\beta,\beta'=1}^n\overline{e}_\beta(f_1)\overline{e}_{\beta'}(f_2)\overline{e}_{k}(lnf)\langle \overline{e}_\beta^*,\nabla_{\overline{e}_{k'}}^{F,*}\overline{e}_{\beta'}^*\big\rangle_F\nonumber\\
&-\frac{1}{f^3}\sum_{\alpha'=1}^m\sum_{k=1}^n\overline{e}_k(f_1)e_{\alpha'}(f_2)e_{\alpha'}(f)\overline{e}_{k}(lnf)-\frac{1}{f^4}\sum_{k,\beta'=1}^n\overline{e}_k(f_1)\overline{e}_{k}\big(\overline{e}_{\beta'}(f_2)\big)\overline{e}_{\beta'}(lnf)\nonumber\\
&+\frac{1}{f^4}\sum_{k,\beta=1}^n\overline{e}_\beta(f_1)\overline{e}_k(f_2)\overline{e}_{k}(lnf)\overline{e}_k(lnf)+\frac{1}{f^4}\sum_{k=1}^n\overline{e}_k(f_1)\overline{e}_k(f_2)\big|{\rm grad}_F(lnf)\big|^2_F\nonumber\\
&-\frac{1}{f^5}\sum_{k,\beta'=1}^n\overline{e}_k(f_1)\overline{e}_{\beta'}(f_2)\overline{e}_{\beta'}(f)\overline{e}_{k}(lnf)-\frac{1}{f^4}\sum_{k,\beta'=1}^n\overline{e}_k(f_1)\overline{e}_{\beta'}(f_2)\big\langle\big({\rm grad}_F(lnf)\big)^*,\nabla_{\overline{e}_{k}}^{F,*}\overline{e}_{\beta'}^*\big\rangle_F\nonumber\\
&-\frac{1}{f^3}\sum_{\alpha'=1}^m\sum_{k=1}^n\overline{e}_k(f_1)\overline{e}_{k}\big(e_{\alpha'}(f_2)\big)e_{\alpha'}(f)+\frac{1}{f^4}\sum_{k=1}^n\overline{e}_k(f_1)\overline{e}_k(f_2)\big|{\rm grad}_Bf\big|^2_B\nonumber\\
&+\frac{1}{f^4}\sum_{\alpha'=1}^m\sum_{\beta=1}^n\overline{e}_\beta(f_1)e_{\alpha'}(f_2)e_{\alpha'}(f)\overline{e}_\beta(f)+\frac{1}{f^5}\sum_{k,\beta=1}^n\overline{e}_\beta(f_1)\overline{e}_{k}\big(\overline{e}_k(f_2)\big)\overline{e}_{\beta}(f)\nonumber\\
&-\frac{1}{f^5}\sum_{k,\beta=1}^n\overline{e}_\beta(f_1)\overline{e}_k(f_2)\overline{e}_{k}(lnf)\overline{e}_\beta(f)-\frac{1}{f^5}\sum_{k,\beta=1}^n\overline{e}_\beta(f_1)\overline{e}_k(f_2)\overline{e}_{\beta}(f)\overline{e}_{k}(lnf)\nonumber\\
&+\frac{1}{f^6}\sum_{\beta,\beta'=1}^n\overline{e}_\beta(f_1)\overline{e}_{\beta'}(f_2)\overline{e}_{\beta}(f)\overline{e}_{\beta'}(f)
+\frac{1}{f^5}\sum_{k,\beta,\beta'=1}^n\overline{e}_\beta(f_1)\overline{e}_{\beta'}(f_2)\overline{e}_{\beta}(f)\big\langle \overline{e}_k^*,\nabla_{\overline{e}_{k}}^{F,*}\overline{e}_{\beta'}^*\big\rangle_F\nonumber\\
&+\frac{1}{f^3}\sum_{\alpha'=1}^m\sum_{k,\beta=1}^n\overline{e}_\beta(f_1)e_{\alpha'}(f_2)e_{\alpha'}(f)\big\langle \nabla_{\overline{e}_{k}}^{F,*}\overline{e}_{\beta}^*,\overline{e}_{k}^*\big\rangle_F
+\frac{1}{f^4}\sum_{k,\beta,\beta'=1}^n\overline{e}_\beta(f_1)\overline{e}_{k}\big(\overline{e}_{\beta'}(f_2)\big)\big\langle \nabla_{\overline{e}_{k}}^{F,*}\overline{e}_{\beta}^*,\overline{e}_{\beta'}^*\big\rangle_F\nonumber\\
&-\frac{1}{f^4}\sum_{k,\beta,\beta'=1}^n\overline{e}_\beta(f_1)\overline{e}_{\beta'}(f_2)\overline{e}_{k}(lnf)\big\langle \nabla_{\overline{e}_{k}}^{F,*}\overline{e}_{\beta}^*,\overline{e}_{\beta'}^*\big\rangle_F
-\frac{1}{f^4}\sum_{k,\beta=1}^n\overline{e}_\beta(f_1)\overline{e}_k(f_2)\big\langle \big(\nabla_{\overline{e}_{k}}^{F,*}\overline{e}_{\beta}^*,{\rm grad}_F(lnf)\big)^*\big\rangle_F\nonumber\\
&+\frac{1}{f^5}\sum_{k,\beta,\beta'=1}^n\overline{e}_\beta(f_1)\overline{e}_{\beta'}(f_2)\overline{e}_{\beta'}(f)\big\langle \nabla_{\overline{e}_{k}}^{F,*}\overline{e}_{\beta}^*,\overline{e}_{k}^*\big\rangle_F
+\frac{1}{f^4}\sum_{k,\beta,\beta'=1}^n\overline{e}_\beta(f_1)\overline{e}_{\beta'}(f_2)\big\langle \nabla_{\overline{e}_{k}}^{F,*}\overline{e}_{\beta}^*,\nabla_{\overline{e}_{k}}^{F,*}\overline{e}_{\beta'}^*\big\rangle_F\nonumber\\
:=&c_3.
\end{align}
\noindent (4 )By (\ref{a8}), a simple calculation gives
\begin{align}\label{a12}
-\frac{1}{2}\Delta f_1\Delta f_2=&-\frac{1}{2}\Big(\Delta_Bf_1+\frac{1}{f^2}\Delta_Ff_1+\frac{2-n}{f^3}{\rm grad}_{F}f(f_1)-\frac{n}{f}{\rm grad}_{B}f(f_1)\Big)\nonumber\\
&\times\Big(\Delta_Bf_2+\frac{1}{f^2}\Delta_Ff_2+\frac{2-n}{f^3}{\rm grad}_{F}f(f_2)-\frac{n}{f}{\rm grad}_{B}f(f_2)\Big)\nonumber\\
=&-\frac{1}{2}\bigg\{\Delta_Bf_1\Delta_Bf_2+\frac{1}{f^2}\big(\Delta_Bf_1\Delta_Ff_2+\Delta_Bf_2\Delta_Ff_1\big)+\frac{1}{f^4}\Delta_Ff_1\Delta_Ff_2\nonumber\\
&+\Big(\frac{2-n}{f^5}{\rm grad}_{F}f(f_2)-\frac{n}{f^3}\big(f^2\Delta_B+\Delta_F\big)(f_1){\rm grad}_{B}f(f_2)\Big)\big(f^2\Delta_B+\Delta_F\big)(f_1)\nonumber\\
&+\Big[\frac{2-n}{f^5}\big(f^2\Delta_B+\Delta_F\big)(f_1)-\frac{n(2-n)}{f^4}{\rm grad}_{B}f(f_2)+\frac{(2-n)^2}{f^6}{\rm grad}_{F}f(f_2)\Big]{\rm grad}_{F}f(f_1)\nonumber\\
&-\Big[\frac{n}{f^3}\big(f^2\Delta_B+\Delta_F\big)(f_2)-\frac{n^2}{f^2}{\rm grad}_{B}f(f_2)+\frac{n(2-n)}{f^4}{\rm grad}_{F}f(f_2)\Big]{\rm grad}_{B}f(f_1)\bigg\}\nonumber\\
:=&c_4.
\end{align}
Summing up (\ref{a9}), (\ref{a10}), (\ref{a11}) and (\ref{a12}), we obtain
\begin{prope}
For $M=B\times_{f}F$, $dimM=4$, then we have
\begin{align}
\Omega_4(f_1,f_2)_g=[c_1+c_2+c_3+c_4]f^ndvol_Bdvol_F.
\end{align}
\end{prope}
\indent Next, for $n=2$, we will compute the sub-Riemannian limit of $$\frac{1}{3}r\langle df_1,df_2\rangle ,~~\Delta\langle df_1,df_2\rangle ,~~\langle \nabla df_1,\nabla df_2\rangle ,~~-\frac{1}{2}\Delta f_1\Delta f_2$$ respectively.\\

\noindent (1) By direct computations, we have
\begin{align}
\langle d_Ff_1,d_Ff_2\rangle _F=&\big\langle \overline{e}_1^L(f_1)\overline{e}_1^{*,L}+\overline{e}_2^L(f_1)\overline{e}_2^{*,L},\overline{e}_1^L(f_2)\overline{e}_1^{*,L}+\overline{e}_2^L(f_2)\overline{e}_2^{*,L} \big\rangle _F\nonumber\\
=&\overline{e}_1^L(f_1)\overline{e}_1^L(f_2)+\overline{e}_2^L(f_1)\overline{e}_2^L(f_2).
\end{align}
Then
\begin{align}
{\rm lim}_{L\rightarrow +\infty}\langle d_Ff_1,d_Ff_2\rangle _F=\overline{e}_1(f_1)\overline{e}_1(f_2).
\end{align}
Therefore, we get
\begin{align}\label{a16}
&{\rm lim}_{L\rightarrow +\infty}\frac{1}{\sqrt{L}}\int_{B\times F^{L}}f_0\Big(\frac{1}{3}r\langle df_1,df_2\rangle\Big)dvol_{B}dvol_{L}\nonumber\\
=&\frac{1}{3}{\rm lim}_{L\rightarrow +\infty}\frac{1}{\sqrt{L}}\int_{B\times F^{L}}f_0\Big(r\langle df_1,df_2\rangle\Big)dvol_{B}dvol_{L}\nonumber\\
=&\frac{1}{3}\int_{B\times F}f_0\bigg[\Big(f^2S^{B}+2\widetilde{l}f\triangle_Bf+A_1+\widetilde{l}(\widetilde{l}-1)|{\rm grad}_Bf|^2_B\Big)\langle d_Bf_1,d_Bf_2\rangle_B\nonumber\\
&+\Big(S^{B}+\frac{2\widetilde{l}}{f}\triangle_Bf+\frac{A_1}{f^2}+\widetilde{l}(\widetilde{l}-1)\frac{|{\rm grad}_Bf|^2_B}{f^2}\Big)\overline{e}_1(f_1)\overline{e}_1(f_2)\bigg]dvol_{B}d\sigma_{\Sigma}\nonumber\\
:=&d_1.
\end{align}
\noindent (2)
\begin{align}
\Delta_F=-\overline{e}_1^L\overline{e}_1^L-\overline{e}_2^L\overline{e}_2^L+\nabla_{\overline{e}_1}^L\overline{e}_1+\nabla_{\overline{e}_2}^L\overline{e}_2.
\end{align}
An easy calculation gives
\begin{align}
\nabla_{\overline{e}_1}^L\overline{e}_1=&\big\langle \nabla_{\overline{e}_1}^L\overline{e}_1,\overline{e}_1 \big\rangle_L\overline{e}_1^L+\big\langle \nabla_{\overline{e}_1}^L\overline{e}_1,\overline{e}_2 \big\rangle_L\overline{e}_2^L\nonumber\\
=&\big\langle \nabla_{\overline{e}_1}^L\overline{e}_1,\overline{e}_2 \big\rangle_L\overline{e}_2^L.
\end{align}
Similarly,
\begin{align}
\nabla_{\overline{e}_2}^L\overline{e}_2=\big\langle \nabla_{\overline{e}_2}^L\overline{e}_2,\overline{e}_1 \big\rangle_L\overline{e}_1^L.
\end{align}
So
\begin{align}\label{a14}
{\rm lim}_{L\rightarrow +\infty}\Delta_F=-\overline{e}_1\overline{e}_1-\frac{2\tau X_3u}{l}\overline{e}_1.
\end{align}
Also,
\begin{align}
{\rm grad}_{F}f=&\big\langle {\rm grad}_{F}f,\overline{e}_1 \big\rangle_L\overline{e}_1^L+\big\langle {\rm grad}_{F}f,\overline{e}_2 \big\rangle_L\overline{e}_2^L\nonumber\\
=&\overline{e}_1^L(f)\overline{e}_1^L+\overline{e}_2^L(f)\overline{e}_2^L.
\end{align}
Then
\begin{align}\label{a15}
{\rm lim}_{L\rightarrow +\infty}{\rm grad}_{F}f=\overline{e}_1(f)\overline{e}_1.
\end{align}
Thus
\begin{align}\label{a17}
&{\rm lim}_{L\rightarrow +\infty}\frac{1}{\sqrt{L}}\int_{B\times F^{L}}f_0\Big(\Delta\langle df_1,df_2\rangle\Big)dvol_{B}dvol_{L}\nonumber\\
=&\int_{B\times F}f_0\bigg[\bigg(f^2\Delta_{B}-\overline{e}_1\overline{e}_1-\frac{2\tau X_3u}{l}\overline{e}_1+\frac{2-n}{f}\overline{e}_1(f)\overline{e}_1-nf{\rm grad}_Bf\bigg)\langle d_Bf_1,d_Bf_2\rangle _B\nonumber\\
&+\bigg(f^2\Delta_{B}(\frac{1}{f^2})\overline{e}_1(f_1)\overline{e}_1(f_2)+\Delta_{B}\big(\overline{e}_1(f_1)\overline{e}_1(f_2)\big)-2f^2\big\langle {\rm grad}_B(\frac{1}{f^2}),{\rm grad}_B\big(\overline{e}_1(f_1)\overline{e}_1(f_2)\big) \big\rangle_B\nonumber\\
&-\Big(\overline{e}_1\big(\overline{e}_1(\frac{1}{f^2})\big)+\frac{2\tau X_3u}{l}\overline{e}_1(\frac{1}{f^2})\Big)\overline{e}_1(f_1)\overline{e}_1(f_2)-\frac{1}{f^2}\big(\overline{e}_1\overline{e}_1+\frac{2\tau X_3u}{l}\overline{e}_1\big)\big(\overline{e}_1(f_1)\overline{e}_1(f_2)\big)\nonumber\\
&-2\big\langle \overline{e}_1(\frac{1}{f^2})\overline{e}_1,\overline{e}_1\big(\overline{e}_1(f_1)\overline{e}_1(f_2)\big)\overline{e}_1 \big\rangle+\frac{2-n}{f}\overline{e}_1(f)\overline{e}_1(\frac{1}{f^2})\overline{e}_1(f_1)\overline{e}_1(f_2)+\frac{2-n}{f^3}\overline{e}_1(f)\overline{e}_1\big(\overline{e}_1(f_1)\overline{e}_1(f_2)\big)\nonumber\\
&-nf{\rm grad}_Bf(\frac{1}{f^2})\overline{e}_1(f_1)\overline{e}_1(f_2)-\frac{n}{f}{\rm grad}_Bf\big(\overline{e}_1(f_1)\overline{e}_1(f_2)\bigg)\bigg]dvol_{B}d\sigma_{\Sigma}\nonumber\\
:=&d_2.
\end{align}
\noindent (3) By computations, we have
\begin{align}\label{a13}
\sum_{k,k',\beta'=1}^2\big\langle \nabla_{\overline{e}_{k'}}^{F,*}\overline{e}_{\beta'}^*,\overline{e}_{k}^*\big\rangle_F=&\big\langle \nabla_{\overline{e}_{1}}^{F,*}\overline{e}_{1}^*,\overline{e}_{1}^*\big\rangle_F+\big\langle \nabla_{\overline{e}_{1}}^{F,*}\overline{e}_{1}^*,\overline{e}_{2}^*\big\rangle_F+\big\langle \nabla_{\overline{e}_{1}}^{F,*}\overline{e}_{2}^*,\overline{e}_{1}^*\big\rangle_F+\big\langle \nabla_{\overline{e}_{1}}^{F,*}\overline{e}_{2}^*,\overline{e}_{2}^*\big\rangle_F\nonumber\\
&+\big\langle \nabla_{\overline{e}_{2}}^{F,*}\overline{e}_{1}^*,\overline{e}_{1}^*\big\rangle_F+\big\langle \nabla_{\overline{e}_{2}}^{F,*}\overline{e}_{1}^*,\overline{e}_{2}^*\big\rangle_F+\big\langle \nabla_{\overline{e}_{2}}^{F,*}\overline{e}_{2}^*,\overline{e}_{1}^*\big\rangle_F+\big\langle \nabla_{\overline{e}_{2}}^{F,*}\overline{e}_{2}^*,\overline{e}_{2}^*\big\rangle_F\nonumber\\
=&-\Big(\big\langle \nabla_{\overline{e}_{1}}^{F}\overline{e}_{1},\overline{e}_{2}\big\rangle_L+\big\langle \nabla_{\overline{e}_{2}}^{F}\overline{e}_{1},\overline{e}_{2}\big\rangle_L\Big)|\overline{e}_{1}|^2-\Big(\big\langle \nabla_{\overline{e}_{1}}^{F}\overline{e}_{2},\overline{e}_{1}\big\rangle_L+\big\langle \nabla_{\overline{e}_{2}}^{F}\overline{e}_{2},\overline{e}_{1}\big\rangle_L\Big)|\overline{e}_{2}|^2.
\end{align}
By further calculation, we obtain
\begin{align}
{\rm lim}_{L\rightarrow +\infty}\sum_{k,k',\beta'=1}^2\big\langle \nabla_{\overline{e}_{k'}}^{F,*}\overline{e}_{\beta'}^*,\overline{e}_{k}^*\big\rangle_F=\frac{\tau X_3u}{l}.
\end{align}
Moreover
\begin{align}
\big|{\rm grad}_F(lnf)\big|^2_F=&g\big\langle {\rm grad}_F(lnf),{\rm grad}_F(lnf)\big\rangle_F\nonumber\\
=&g\big\langle\overline{e}_1^L(lnf)\overline{e}_1^L+\overline{e}_2^L(lnf)\overline{e}_2^L,\overline{e}_1^L(lnf)\overline{e}_1^L+\overline{e}_2^L(lnf)\overline{e}_2^L \big\rangle_L\nonumber\\
=&|\overline{e}_1^L(lnf)|^2+|\overline{e}_2^L(lnf)|^2.
\end{align}
Then
\begin{align}
{\rm lim}_{L\rightarrow +\infty}\big|{\rm grad}_F(lnf)\big|^2_F=|\overline{e}_1(lnf)|^2.
\end{align}
Similar to (\ref{a13}), we also get
\begin{align}
{\rm lim}_{L\rightarrow +\infty}\sum_{k',\beta'=1}^2\big\langle\big({\rm grad}_F(lnf)\big)^*,\nabla_{\overline{e}_{k'}}^{F,*}\overline{e}_{\beta'}^*\big\rangle_F=-\overline{e}_1(lnf)\frac{\tau X_3u}{l},
\end{align}
and
\begin{align}
{\rm lim}_{L\rightarrow +\infty}\sum_{k,\beta,\beta'=1}^2\big\langle \nabla_{\overline{e}_{k}}^{F,*}\overline{e}_{\beta}^*,\nabla_{\overline{e}_{k}}^{F,*}\overline{e}_{\beta'}^*\big\rangle_F=5\big(\frac{\tau X_3u}{l}\big)^2.
\end{align}
Thus
\begin{align}\label{a18}
&{\rm lim}_{L\rightarrow +\infty}\frac{1}{\sqrt{L}}\int_{B\times F^{L}}f_0\Big(\langle \nabla df_1,\nabla df_2\rangle\Big)dvol_{B}dvol_{L}\nonumber\\
=&\int_{B\times F}f_0\bigg(f^2\sum_{j,\alpha,\alpha'=1}^m\Big[\Big(e_{j}\big(e_{\alpha}(f_2)\big)
+e_{\alpha'}(f_2)\big\langle e_\alpha^*,\nabla_{e_{j}}^{B,*}e_{\alpha'}^*\big\rangle\Big)e_j\big(e_\alpha(f_1)\big)+\Big(e_{j}\big(e_{\alpha'}(f_2)\big)\nonumber\\
&\times\big\langle \nabla_{e_{j}}^{B,*}e_{\alpha}^*,e_{\alpha'}^*\big\rangle
+e_{\alpha'}(f_2)\big\langle \nabla_{e_{j}}^{B,*}e_{\alpha}^*,\nabla_{e_{j}}^{B,*}e_{\alpha'}^*\big\rangle\Big)e_\alpha(f_1)\Big]+\frac{1}{f^2}\sum_{j=1}^m\Big(fe_j\big(\overline{e}_1(f_1)\big)\nonumber\\
&-\overline{e}_1(f_1)e_{j}(f)\Big)\Big(f e_{j}\big(\overline{e}_{1}(f_2)\big)-\overline{e}_{1}(f_2)e_{j}(f)\Big)+\sum_{\alpha=1}^m\overline{e}_1\big(e_\alpha(f_1)\big)\overline{e}_{1}\big(e_{\alpha}(f_2)\big)\nonumber\\
&+\frac{1}{f^2}\sum_{\alpha'=1}^m\Big(f\overline{e}_1\big(\overline{e}_1(f_1)\big)e_{\alpha'}(f_2)+f\overline{e}_1(f_1)\overline{e}_1\big(e_{\alpha'}(f_2)\big)-2f\overline{e}_1(lnf)\overline{e}_1(f_1)e_{\alpha'}(f_2)\nonumber\\
&+\overline{e}_1(f)\overline{e}_1(f_1)e_{\alpha'}(f_2)+f\frac{\tau X_3u}{l}\overline{e}_1(f_1)e_{\alpha'}(f_2)\Big)e_{\alpha'}(f)-\frac{1}{f^3}\Big(2f\overline{e}_1(lnf)-\overline{e}_1(f)-f\frac{\tau X_3u}{l}\Big)\nonumber\\
&\times\Big(\overline{e}_1\big(\overline{e}_1(f_1)\big)\overline{e}_1(f_2)+\overline{e}_1(f_1)\overline{e}_1\big(\overline{e}_1(f_2)\big)\Big)+\frac{1}{f^4}\Big(4f^2|\overline{e}_1(lnf)|^2+f^2|{\rm grad}_Bf|^2_B\nonumber\\
&+5\big(f\frac{\tau X_3u}{l}\big)^2+|\overline{e}_1(f)|^2-3f\overline{e}_1(lnf)\overline{e}_1(f)+2f\frac{\tau X_3u}{l}\overline{e}_1(f)\Big)\overline{e}_1(f_1)\overline{e}_1(f_2)\bigg)dvol_{B}d\sigma_{\Sigma}\nonumber\\
:=&d_3.
\end{align}
\noindent (4) By (\ref{a14}) and (\ref{a15}), a simple calculation gives
\begin{align}\label{a19}
&{\rm lim}_{L\rightarrow +\infty}\frac{1}{\sqrt{L}}\int_{B\times F^{L}}f_0\Big(-\frac{1}{2}\Delta f_1\Delta f_2\Big)dvol_{B}dvol_{L}\nonumber\\
=&-\frac{1}{2}{\rm lim}_{L\rightarrow +\infty}\frac{1}{\sqrt{L}}\int_{B\times F^{L}}f_0\Big(\Delta f_1\Delta f_2\Big)dvol_{B}dvol_{L}\nonumber\\
=&-\frac{1}{2}\int_{B\times F}f_0\bigg(f^2\Delta_Bf_1\Delta_Bf_2-\Big(\Delta_Bf_1\overline{e}_1\big(\overline{e}_1(f_2)\big)+\Delta_Bf_1\frac{2\tau X_3u}{l}\overline{e}_1(f_2)+\Delta_Bf_2\overline{e}_1\big(\overline{e}_1(f_1)\big)\nonumber\\
&+\Delta_Bf_2\frac{2\tau X_3u}{l}\overline{e}_1(f_1)\Big)+\frac{1}{f^2}\Big(\overline{e}_1\big(\overline{e}_1(f_1)\big)+\frac{2\tau X_3u}{l}\overline{e}_1(f_1)\Big)\Big(\overline{e}_1\big(\overline{e}_1(f_2)\big)+\frac{2\tau X_3u}{l}\overline{e}_1(f_2)\Big)\nonumber\\
&+\Big(\frac{2-n}{f^3}\overline{e}_1(f)\overline{e}_1(f_2)-\frac{n}{f}{\rm grad}_Bf(f_2)\Big)\Big(f^2\Delta_B-\overline{e}_1\overline{e}_1-\frac{2\tau X_3u}{l}\overline{e}_1\Big)(f_1)+\Big[\frac{2-n}{f^3}\big(f^2\Delta_B\nonumber\\
&-\overline{e}_1\overline{e}_1-\frac{2\tau X_3u}{l}\overline{e}_1\big)(f_1)-\frac{n(2-n)}{f^2}{\rm grad}_{B}f(f_2)+\frac{(2-n)^2}{f^4}\overline{e}_1(f)\overline{e}_1(f_2)\Big]\overline{e}_1(f)\overline{e}_1(f_1)\nonumber\\
&-\Big[\frac{n}{f}\big(f^2\Delta_B-\overline{e}_1\overline{e}_1-\frac{2\tau X_3u}{l}\overline{e}_1\big)(f_2)-n^2{\rm grad}_{B}f(f_2)+\frac{n(2-n)}{f^2}\overline{e}_1(f)\overline{e}_1(f_2)\Big]{\rm grad}_{B}f(f_1) \bigg)dvol_{B}d\sigma_{\Sigma}\nonumber\\
:=&d_4.
\end{align}
Summing up (\ref{a16}), (\ref{a17}), (\ref{a18}) and (\ref{a19}), we obtain
\begin{align}
{\rm lim}_{L\rightarrow +\infty}\frac{1}{\sqrt{L}}\int_Mf_0\Omega_4(f_1,f_2)_g=d_1+d_2+d_3+d_4.
\end{align}
\indent Hence, we obtain the following theorem.
\begin{thm}
Let $M=B\times_{f}F^L$ be 4-dimensional twisted BCV spaces, suppose that the characteristic set $C(F)$ satisfies $\mathcal{H}^1(C(F))=0$ and that $||\nabla_Hu||_H^{-1}$ is locally summable with respect to the Euclidean $2$-dimensional Hausdorff measure near the characteristic set $C(F)$, then the following equality holds:
\begin{align}
{\rm lim}_{L\rightarrow +\infty}\frac{1}{\sqrt{L}}{\rm Wres}\big(f_0[\widetilde{F},f_1][\widetilde{F},f_2]\big)=d_1+d_2+d_3+d_4
\end{align} which 
is a Hochschild 2-cocyle on $C^{\infty}(M)$.
\end{thm}

\section*{Data availability}
No data was used for the research described in the article.

\section*{Acknowledgements}
This work was supported by NSFC No. 11771070. 

\end{document}